\documentclass{article}
\usepackage[utf8]{inputenc}
\usepackage{amsmath}
\usepackage{amsthm}
\usepackage{amsfonts}
\usepackage{amssymb}
\usepackage{bm}
\usepackage{xcolor}
\usepackage[ruled,vlined,noend]{algorithm2e}
\title{Convergence Analysis of the Algorithm in
"Efficient and Robust Discrete Conformal Equivalence with Boundary" \protect\cite{Campen:2021}}
\author{Denis Zorin}
\date{August 2021}
\newcommand{\bbR}{\mathbb{R}}
\newcommand*{\ve}[1]{\bm{#1}}
\newcommand{\df}{\nabla f}
\newcommand{\ddf}{\nabla^2 f}
\newcommand{\dfu}{\nabla f(u)}
\newcommand{\ddfu}{\nabla^2 f(u)}
\newcommand{\tf}{\tilde{f}}
\newcommand{\pg}{\tilde{g}}
\newcommand{\uk}{u^{(k)}}
\newcommand{\ukp}{u^{(k+1)}}
\newcommand{\uinit}{u^{(0)}}

\newcommand{\lmsq}{\lambda^2}

\newcommand{\cE}{{\cal E}}

\newtheorem{thm}{Theorem}
\begin{document}

\maketitle
\section{Newton method with energy-free line search} 
We start with analyzing the general form of the algorithm  used in \cite{Campen:2021}:

\begin{algorithm}[h!]
\SetAlgoLined
\DontPrintSemicolon
\SetKwInOut{Input}{Input}
\SetKwInOut{Output}{Output}
\SetKwProg{Fn}{Function}{:}{}
\SetKwRepeat{Do}{do}{while}
\SetKw{Not}{not}
\Fn{{\scshape{LineSearchNewton}}$(f,\uinit )$}
{
 $u\gets \uinit$\;
 $g \gets g(\uinit)$\;
 \While{$\|g\| > \varepsilon$ }{
   $g \gets g(u)$\tcp*{gradient}
   $H \gets H(u)$\tcp*{Hessian}
   $d \gets -H^{-1}g$\tcp*{Newton direction}
   $u \gets$ u+\scshape{LineSearch}$(u, d)d$ \tcp*{Newton step}
 }
 \Return $u$\;
}

\vspace{4pt}

\Fn{\scshape{LineSearch}$( \ve{u}, \ve{d})$}
{
   \If{$\frac{1}{2}\left( d^T g(u+d/2) +  d^T g( u+d) \right) \leq \alpha d^Tg(u)$ }
    { 
    \Return $1$\;
    }
    $t \gets 1/2$\;
    \While{true}{
    \If{$d^T g(u+td) \leq 0$ }
    { 
    \Return $t$\;
    }
    $t \gets \frac{1}{2}t$ \tcp*{backtracking line search}}
}
 \caption{Newton minimization with energy-free line search\label{alg:Newton}}
\end{algorithm}
In this algorithm, $\alpha$ is chosen in a way similar to the constant in the Armijo  
condition. This Armijo-like condition, not relying on function evaluation, is used at the first iteration of the line search---it could also be used on all iterations, but one expects a slightly 
faster convergence at the initial (damped) Newton phase without it, before
the quadratic convergence regime is reached.


Define 
\[
S = \{ u \in \bbR^n | f(u) \leq f(u^{(0)})\}\]
where $u^{(0)}$ is the starting point used by the algorithm,
i.e., $S$ is the sublevel set of $f(u)$ for $f(u^{0})$. 
We assume this set is bounded (we will show this is the case for the specific function that appears in our equations). It is also convex, as the function $f$
is convex. 

Below, we use $\|\cdot\|$ to denote the standard vector $\ell_2$-norm. 
\begin{thm}
Suppose a convex function $f(u): \bbR^n \rightarrow \bbR$
is twice differentiable and its Hessian has singular values bounded by $m$ and $M$ on $S$ from below and above respectively (i.e., is strongly convex on $S$), and the Hessian is Lipschitz  on $S$ with constant $L$,
i.e., 
\[
\| \nabla^2 f(u) - \nabla^2 f(w)\| \leq L\|u-w\|
\]

Then Algorithm~\ref{alg:Newton} converges to the (unique) minimum of the function, and the rate of convergence is quadratic after an iteration $k_0(m, M, L)$.\label{th:general}
\end{thm}
\begin{proof}
The proof largely follows the proof of convergence of Newton's method with the standard backtracking line search, using Armijo's condition, which requires evaluation of the function $f$. In comparison, the algorithm above uses only its gradient and Hessian.

The proof proceeds in three steps. 

\begin{enumerate} 
\item First, we show that for any $\eta > 0$, if $\|\df(\uk)\| \geq \eta$, then there is a $\gamma(\eta) > 0$ such that
\begin{equation}
f(\ukp) - f(\uk) \leq -\gamma
\end{equation}
As the function value  cannot go below the minimum value $p^*$, 
the number of iterations required to reach any value of gradient is finite, bounded by $|f(\uinit) - p^*|/\gamma$. 
\item Next, we show that there is a choice of $\eta > 0$, such that our line search always chooses the full Newton step ($t=1$) if 
$\|\df(x)\| \leq \eta$, i.e., after a finite number of steps $k_0(m,M,L)$, we have a standard Newton method with quadratic convergence. 
\item The remainder of the argument is exactly the same as in 
\cite[Sec. 9.5.3]{boyd2004convex}: one shows that for all steps $k \geq k_0$, $\|\df(x)\| \leq \eta$, and infer quadratic convergence. 
\end{enumerate}
We only need to prove the first two steps, as these depend on the specific choice of the line search. 
Define the \emph{Newton decrement} as 
\[
\lambda^2(u) = \dfu^T\ddfu^{-1}\dfu = -\dfu^T d = d^T \ddfu d, 
\] 
where $d = -\ddf^{-1}\df$ is the Newton search direction. 
Below, we drop explicit dependence on $u$ for brevity, and only explicitly indicate the dependence on values different from 
$u$, e.g., $u + td$.

\paragraph{Damped Newton phase.} 
Suppose $\|\df\| \geq \eta$; then it follows from the bounds
$m\|v\|^2 \leq v^T \ddf v \leq M\|v\|^2$, that 
\[
f(u+td) \leq f + t \dfu^T d + \frac{M}{2}\|d\|^2 t^2
\]
As $\lmsq = d^T \ddf d \geq m \|d\|^2$, then $\|d\|^2 \leq  \lmsq/m$,
\[
f(u+td) \leq f - t\lmsq + \frac{M}{2m}\lmsq t^2
\]
Denote $\pg(t) = d^T \df(u+td)$, and 
$\tf(t) = f(u + td)$.
Consider first the case when the first line search 
exit condition is satisfied,  i.e., $\pg(1) + \pg(\frac{1}{2}) \leq 2\alpha \pg(0)$.
Observe that as $\pg(0) < 0$ and $\pg(t)$ is increasing, (Newton direction is a descent direction), the  line search inequality can only hold
if $\pg(\frac{1}{2}) < 0$; if $\pg(1) \leq 0$, then the decrease of the function is at least $\int_0^1 \pg(t)dt \leq \pg(1)$. If $\pg(1) > 0$, then 
the decrease  is estimated by 
\[
\begin{split}
\tf(t) - \tf(0) &= \int_0^{\frac{1}{2}} \pg(s) ds + \int_{\frac{1}{2}}^1 \pg(s) ds  \leq  \int_0^{\frac{1}{2}} \max_{[0,\frac{1}{2}]} \pg(s) ds + \int_{\frac{1}{2}}^1 \max_{[\frac{1}{2},1]} \pg(s) ds \\
&= \frac{1}{2}(\pg(\frac{1}{2}) + \pg(1)) \leq \alpha \pg(0) = -\alpha \lmsq \leq -\alpha \frac{1}{M} \|
\df\|^2 \leq -\frac{\alpha \eta^2}{M}
\end{split}
\]
as  $\pg(s)$ is increasing on the interval due to strong convexity. 
Next, consider the case when the second termination condition in the 
line search is satisfied but not the first, i.e., the value returned is $t$ for which $\pg(t) < 0$. 
In this case, Let $t'$ be the point along the search direction for which  $\pg(t') = 0$, i.e., the point where $f(u+td)$ is minimal;  $t' < 1$ because it follows from the failure of the first condition that 
$\pg(1) > 0$.
Then $t'/2 \leq t \leq t'$, as we half the value of t at each step of the line search. As $\tf(t)$ is decreasing up to $t'$, 
$\tf(t) \leq \tf(t'/2)$. Observe that by convexity 
$\tf(t) \leq \tf(t'/2) \leq (\tf(0) + \tf(t'))/2$, i.e., the function decreases at $t$ at least by half of the decrease 
at the minimum.  We can estimate the minimum value using
strong convexity and finding the minimum of the bounding quadratic function explicitly:  

\[
\min_{[0,1]} \tf(t) \leq \min_{[0,1]} \tf(0) - t\lmsq + \frac{M}{2m}t^2\lambda^2 = \tf(0) - \frac{m \lmsq}{2M} \leq  \tf(0) -\frac{m}{2M^2}{\eta^2}
\]
We conclude that at every step the energy decreases at 
least by $\eta^2 \min(\frac{m}{4M^2}, \frac{\alpha}{M}) = \gamma(\eta)$.
Note that the first line search termination condition 
is not necessary for this to hold. If this condition is omitted, the step is either 1, if $\pg(1) \leq 0$,  or $t' < t$ as before. In the first case, 
the function decreases on $[0,1]$ so the decrease is greater than the 
bound for the decrease at $t = m/M$, which we use in the case when $t' < t$.
The method converges without the first termination condition, but the rate of convergence may be suboptimal. The first condition is needed to prove the second part. 
\paragraph{Transition to quadratically convergent phase.} 
Next, we show that there is a choice of $\eta$, for which our line search is guaranteed to choose step $t=1$.
Starting from the Lipschitz condition on $\ddf$, we obtain
\[
|d^T \ddf(u + td) d - d^T \ddfu d |  \leq t L \|d\|^3 
\]
Note that $\pg'(t) = d^T \ddf(u+td) d $, i.e., 
\[
|\pg'(t) - \pg'(0)| \leq t L \|d \|^3 \leq tL m^{-3/2}\lambda^3
\]
using $m\|d\|^2 \leq \lmsq$.
From this, we obtain 
\[
\pg'(t) \leq  \lmsq +  tL m^{-3/2}\lambda^3
\]
Integrating this expression with respect to $t$, we obtain 
\[
\pg(t) \leq -\lmsq + \lmsq t + \frac{t^2 L}{2} m^{-3/2} \lambda^3.
\]
Evaluating these expressions for $\pg(1)$ and $\pg(\frac{1}{2})$, 
we get for the right-hand side of the first termination condition in the line search: 
\[
\frac{1}{2}\left(\pg(1) + \pg(\frac{1}{2}) \right) \leq \lmsq\left(-\frac{1}{2} + \frac{5}{8} m^{-3/2}\lambda L \right) 
\]
The left-hand side $\alpha g(0)$ is simply $-\alpha\lmsq$.
As $\lmsq \leq \|\ddf^{-1}\|\|\df\|^2 \leq m^{-1}\eta^2$, 
we conclude that for the condition to be satisfied (i.e., for 
the line search to accept the step $t=1$) 
it is sufficient to have 
\[
\left(\frac{1}{2} - \frac{5}{8} m^{-2}\eta L \right) \geq \alpha
\]
From this, we obtain the bound for $\eta$: 
\[
\eta \leq \left(\frac{1}{2} -\alpha \right)\frac{8m^2}{5L}
\]
which is similar to the bound obtained in \cite{boyd2004convex}
for the standard line search with the Armijo condition.
\end{proof}

\begin{thm}
The convex function $\cE(u)$ whose minimum yields the 
conformally equivalent metric with prescribed discrete curvature,
satisfies the conditions of Theorem~\ref{th:general}, with $u$ 
restricted to a subspace $W \subset \bbR^n$ eliminating the nullspace of the Hessian of $\cE$, e.g., using the constraint $u_0 = 0$.
\end{thm} 
\begin{proof}
The key fact from which the statement follows relatively easily, is that the set $S$ for the function $\cE$ is bounded, and therefore compact, for 
any choice of the starting point. As we know that the minimum of $\cE$ on $W$ exists and $\cE$ is strictly convex (as the Hessian is positive definite \cite{springborn2020ideal}) by \cite[Corrollary 8.7.1]{rockafellar2015convex}, its level sets 
are bounded; as $S$ is a sublevel set, it is bounded. 

The second critical observation is that the Hessian entries depend continuously only on the angles of an intrinsic Delaunay triangulation, and these angles are strictly nonzero at all points
of $S$: this follows from the fact that \emph{strict} triangle 
inequality is implied by the (possibly nonstrict) Delaunay condition inequality \cite{penner1987decorated}. 

As eigenvalues of a matrix depend continuously on its coefficients, 
in particular, minimal and maximal eigenvalues are continuous functions on $S$. As $S$ is compact, and for any point $u$ the minimal eigenvalue  $\mu_{min} (\nabla^2\cE(u))$ is positive, 
there is a lower bound $m > 0$ for eigenvalues of the Hessian 
at all points of $S$; in the same way, the upper bound $M$  is defined. 

To establish that $\nabla^2\cE(u)$ is Lipschitz, we observe that 
the derivatives of the entries of the Hessian are derivatives of cotangents of angles of triangles of a Delaunay triangulation; as derivatives of angles are cotangents \cite{springborn2020ideal}, 
the resulting third derivatives are also trigonometric functions of angles. For each Penner cell $P$ (i.e., the subset of $W$ for 
which the Delaunay triangulation used to compute the energy, its gradient and Hessian remains constant), the set $P \cap S$, 
is bounded and compact.  For a fixed Delaunay triangulation, angles depend continuously on $u$, hence the minimal angle $\alpha_{min}(u)$ depends on $u$
continuously as a lower envelope of a set of continuous functions. 
As at all points $u$ of $P \cap S$, which is compact, 
$\alpha_{min}(u) > 0$, we conclude that it is bounded from below on each Penner cell intersecting $S$.  As the total number of Penner cells is finite \cite{springborn2020ideal},   there is a global lower bound on minimal angle of the triangulation on all of $S$, 
and, as a consequence, third derivatives are piecewise continuous, 
with discontinuities only at the cell boundaries, and bounded, i.e., all third derivative value jumps on the Penner cell boundaries are bounded. We conclude that the function is Lipschitz. 
\end{proof}
\paragraph{Why do we need the extra condition in the line search?} 
The algorithm converges without this condition; it is only needed to prove that the algorithm converges quadratically (eventually, after some, possibly large, number of iterations in the damped phase).
An essential requirement for quadratic convergence is that a full Newton step is taken when the point is close to the minimum. 
However, the  line search that terminates only when $\pg(t) < 0$ does not guarantee this in general (although we observe this in practice for the specific energy we consider). 
A simple counterexample is the function $f(t) = x^2 + \epsilon x^3$.
A full Newton step will always  attempt to go to zero (as this is the minimum of the quadratic part), where the gradient of the function is positive, so the full step will be rejected, even if the line search starting point is very close to zero. Instead, close enough to zero, the second step size $t=\frac{1}{2}$ will always be accepted, so the convergence will be slower.
\bibliographystyle{plain}
\bibliography{main.bib}
\end{document}